\documentclass[letterpaper,11pt]{amsart}

\usepackage[margin=1.2in]{geometry}
\usepackage{amsmath,amsthm,amssymb}
\usepackage{xspace,xcolor}
\usepackage[breaklinks,colorlinks,citecolor=teal,linkcolor=teal,urlcolor=teal,pagebackref,hyperindex]{hyperref}
\usepackage[alphabetic]{amsrefs}
\usepackage[all]{xy}

\theoremstyle{plain}
\newtheorem{thm}{Theorem}[section]

\newtheorem{lem}[thm]{Lemma}
\newtheorem{prop}[thm]{Proposition}
\newtheorem{cor}[thm]{Corollary}

\theoremstyle{definition}
\newtheorem{defi}[thm]{Definition}

\newtheorem{eg}[thm]{Example}

\theoremstyle{remark}
\newtheorem{rmk}[thm]{Remark}

\newtheorem*{ack}{Acknowledgments}

\def\Z{{\mathbb Z}}

\def\R{{\mathbb R}}

\def\A{{\mathbb A}}

\def\cO{\mathcal{O}}

\def\O{\mathcal{O}}

\def\fa{\mathfrak{a}}
\def\fb{\mathfrak{b}}
\def\fc{\mathfrak{c}}

\def\fm{\mathfrak{m}}

\def\fra{\mathfrak{a}}

\def\frm{\mathfrak{m}}

\def\a{\alpha}

\def\g{\gamma}

\def\f{\phi}
\def\ff{\psi}

\def\p{\pi}

\def\s{\sigma}

\def\Om{\Omega}

\def\.{\cdot}
\def\^{\widehat}
\def\~{\widetilde}
\def\o{\circ}
\def\ov{\overline}

\def\inj{\hookrightarrow}

\def\de{\partial}

\def\({\left(}
\def\){\right)}

\newcommand{\llbracket}{[\negthinspace[}
\newcommand{\rrbracket}{]\negthinspace]}

\renewcommand{\and}{ \ \ \text{ and } \ \ }

\def\Jac{\mathrm{Jac}}

\DeclareMathOperator{\codim} {codim}

\DeclareMathOperator{\Spec} {Spec}

\DeclareMathOperator{\val} {val}

\DeclareMathOperator{\ord} {ord}

\DeclareMathOperator{\Supp} {Supp}

\DeclareMathOperator{\vol} {vol}
\DeclareMathOperator{\ini} {in}

\DeclareMathOperator{\lct} {lct}

\DeclareMathOperator{\Fitt} {Fitt}

\DeclareMathOperator{\jac} {Jac}

\DeclareMathOperator{\Cont} {Cont}

\DeclareMathOperator{\jetcodim} {jet-codim}

\title{The volume of a set of arcs on a variety}

\author{Tommaso de Fernex}

\address{Department of Mathematics, University of Utah, Salt Lake City, UT 48112, USA}
\email{{\tt defernex@math.utah.edu}}

\author{Mircea Musta\c t\u a}

\address{Department of Mathematics, University of Michigan, Ann Arbor, MI 48109, USA}
\email{{\tt mmustata@umich.edu}}

\dedicatory{Dedicated to Lucian B\u{a}descu on the~occasion of
his~seventieth~birthday}

\subjclass[2010]{Primary 13H15, 14E18; Secondary 13A18, 14B05}
\keywords{Arc space, graded sequence of ideals, valuation, Mather discrepancy} 

\thanks{The research of the first author was partially supported by 
NSF grants DMS-1402907 and DMS-1265285.
The research of the second author was partially supported by 
NSF grants DMS-1401227 and DMS-1265256.}

\begin{document}

\maketitle


\section{Introduction}

In this paper, we give a definition of volume for subsets in 
the space of arcs of an algebraic variety and study its properties. 
As our definition implies that the volume of a set of arcs is finite
if and only if its projection to the variety is a finite set of closed points,
we can restrict without loss of generality to the case of an affine variety.
Suppose therefore that $X=\Spec(R)$ is an $n$-dimensional 
affine algebraic variety defined over an algebraically closed field of characteristic zero.
For every ideal $\fa$ in $R$ we denote by $\ell(R/\fa)$ the length of the quotient ring $R/\fa$
and, if the cosupport consists of one point $x$ defined by the ideal $\frm_x$,
we denote by $e(\fa)$ the
Hilbert--Samuel multiplicity of $R_{\fm_x}$ with respect to $\fa R_{\fm_x}$.

Let $X_\infty$ be the arc scheme of $X$. Recall that for every field extension $K/k$, the
$K$-valued points of $X_{\infty}$ are in natural bijection with the arcs
$\Spec K\llbracket t\rrbracket \to X$ (see \cite[Section 3]{EM09}). 
For every subset $C \subseteq X_\infty$ and any integer $m \ge 0$, we consider the ideal 
\[
\fb_m(C) := \{ f \in R \mid \ord_\g(f) \ge m \text{ for all } \g \in C \}.
\]
This defines a graded sequence of ideals $\fb_\bullet(C) = (\fb_m(C))_{m \ge 0}$. 
We then define the volume of $C$ by the formula
\[
\vol(C) := \vol(\fb_\bullet(C)) = \limsup_{m \to \infty} \frac{\ell(R/\fb_m(C))}{m^n/n!}.
\]
It follows from \cite{Cut14} that the limsup is, in fact, a limit. 
It is easy to see that $\vol(C) < \infty$ if and only if the image of $\pi(C)$ in $X$
is a finite set of
closed points.
Here $\p \colon X_\infty \to X$
is the canonical projection mapping an arc $\g$ to its base point $\g(0)$. 
The volume satisfies the following inclusion/exclusion property. 

\begin{prop}
\label{p:inclusion-exclusion}
If $C_1, C_2 \subseteq X_\infty$, then 
\[
\vol(C_1 \cup C_2) + \vol(C_1 \cap C_2) \le \vol(C_1) + \vol(C_2).
\]
\end{prop}

The \emph{contact locus} of order at least $q$ of an ideal $\fa \subseteq R$ is defined to be
\[
\Cont^{\ge q}(\fa) = \{ \g \in X_\infty \mid \ord_\g(\fa) \ge q\}.
\]
Contact loci form a special class of subsets in $X_\infty$. For ideals cosupported at one point,
the volumes of these sets relate to the Samuel multiplicities of the ideal in the following way.

\begin{prop}
\label{p:vol-cont}
For every ideal $\fa \subseteq R$ whose cosupport consists of one point and for every $m,p\geq 1$, we have 
\[
m^n \. \vol(\Cont^{\ge m}(\fa)) \le (mp)^n \. \vol(\Cont^{\ge mp}(\fa)) \le e(\fa)
\]
for every $m,p \ge 1$. Furthermore, both inequalities are equalities for $m$ sufficiently divisible. 
\end{prop}

Generalizing the definition of codimension of a cylinder in the space of arcs of a smooth variety, 
we define the \emph{jet-codimension} of an irreducible closed subset $C$ of $X_\infty$ to be
\[
\jetcodim(C) := \lim_{p \to \infty}\((p+1)n - \dim \ov{\p_p(C)} \)
\]
where $\p_p \colon X_\infty \to X_p$ is the truncation map to the $p$-jet space. 
This definition extends to an arbitrary set $C \subseteq X_\infty$ by taking the 
smallest jet-codimension of the irreducible components of the  closure $\ov C$
of $C$ in $X_\infty$. 
We will see, for instance, that 
if $X$ is smooth, then the jet-codimension of a set $C$ coincides with its Krull codimension $\codim(C)$
(which is similarly defined as the smallest Krull codimension of an irreducible component of $\ov C$).

Our main result relates the volume of a set of arcs
on a Cohen--Macaulay variety to its jet-codimension. 

\begin{thm}
\label{t:vol-codim}
If $X$ is Cohen--Macaulay, of dimension $n$, then 
for every subset $C \subseteq X_{\infty}$ whose image in $X$ is a closed point we have
\[
\vol(C)^{1/n}\.\jetcodim(C) \ge n.
\]
In particular, if $X$ is smooth, then
\[
\vol(C)^{1/n}\.\codim(C) \ge n.
\]
\end{thm}

The proof of this theorem requires a suitable extension of the main result
of \cite{dFEM04} to singular varieties, which we discuss next.  
Let $\fa \subseteq R$ be an $\fm$-primary ideal, where $\fm\subset R$ is a maximal ideal.
If $X$ is smooth, then the colength and the Hilbert--Samuel multiplicity of $\fa$
are related to the log canonical threshold $\lct(\fa)$ by the formulas
\begin{equation}
\label{eq:ell-lct}
(n!\.\ell(\O_X/\fa))^{1/n}\.\lct(\fa) \ge n,
\end{equation}
\begin{equation}
\label{eq:e-lct}
e(\fa)^{1/n}\.\lct(\fa) \ge n.
\end{equation}
We want to extend this result to all Cohen--Macaulay varieties. 
If $X$ is singular, then the log canonical threshold (even when it is defined)
is not the right invariant to consider. Instead, we look at the 
\emph{Mather log canonical threshold}  of the ideal \cite{Ish13}, which is defined by
\[
\^\lct(\fa) := \inf_{f,E} \frac{\ord_E(\jac_f)+1}{\ord_E(\fa)}
\]
where the infimum ranges over all birational morphisms $f \colon Y \to X$, with $Y$ smooth,
and all prime divisors $E \subset Y$, with
$\jac_f$ being the Jacobian ideal of $f$. 

\begin{thm}
\label{t:e-Mather-lct}
With the above notation, if $X$ is Cohen--Macaulay, of dimension $n$, then we have
\begin{equation}
\label{eq:ell-Mather-lct}
(n!\.\ell(\O_X/\fa))^{1/n}\.\^\lct(\fa) \ge n,
\end{equation}
\begin{equation}
\label{eq:e-Mather-lct}
e(\fa)^{1/n}\.\^\lct(\fa) \ge n.
\end{equation}
\end{thm}

The proofs of \eqref{eq:ell-lct} and \eqref{eq:e-lct} rely
on the reduction to monomial ideals via flat degeneration, 
where the inequality follows from Howard's description of log canonical 
thresholds of monomial ideals and the well-known inequality between 
arithmetic and geometric means. 
A slightly more general formulation of \eqref{eq:e-lct} 
is the key ingredient in the proof of a theorem of \cite{dFEM03}
on log canonical thresholds of generic projections.
The proof of Theorem~\ref{t:e-Mather-lct} follows the opposite 
direction: we first prove a theorem on Mather log discrepancies
of generic projections (see Theorem~\ref{t:projection} below), and then deduce
\eqref{eq:ell-Mather-lct} and \eqref{eq:e-Mather-lct} from it.

The paper is organized as follows. In the
next section we prove Theorem~\ref{t:e-Mather-lct}.
Section 3 is devoted to a discussion of volumes of graded sequences of ideals,
with emphasis on sequences associated to pseudo-valuations.
Then, in the 
last section we define the volume of a set of arcs
and prove several properties including those stated above.

\begin{ack}
The results in Section~\ref{s:discr} are
heavily inspired by our work with Lawrence Ein in \cite{dFEM03,dFEM04}, 
and it is a pleasure to thank him for many enlightening
discussions throughout these years.
We would like to thank also Rob Lazarsfeld for useful 
discussions on these topics. 
\end{ack}

\section{Mather log discrepancies}
\label{s:discr}

Let $X$ be a variety of dimension $n$ defined over an algebraically closed field of characteristic zero.
Recall that a \emph{divisor over} $X$ is a prime divisor $E$ on a normal variety $Y$, with a birational morphism
$f\colon Y\to X$. Such a divisor determines a valuation $\ord_E$ of $k(Y)=k(X)$ and as usual, we identify
two divisors over $X$ if they give the same valuation. The valuations that arise in this way are the
\emph{divisorial valuations} of $k(X)$ that have center on $X$ (recall that the center of $\ord_E$ is the closure
of $f(E)$).

Given a birational morphism $f \colon Y \to X$, with $Y$ smooth, we consider
$\jac_f := \Fitt^0(\Om_{Y/X}) \subseteq \O_Y$, the Jacobian ideal of the map. 

\begin{defi}
Given a divisor $E$ over $X$, the  \emph{Mather log discrepancy} $\^a_E(X)$ of $E$ over $X$
is defined as follows. Suppose that $f\colon Y\to X$ is a birational morphism, with $Y$ normal,
such that $E$ is a prime divisor on $Y$. After possibly replacing $Y$ by its smooth locus, we may
assume that $Y$ is smooth. If $\jac_f \subseteq \O_Y$ is the Jacobian ideal of the map,
then
\[
\^a_E(X) := \ord_E(\jac_f)+1.
\]
Given a nonzero ideal sheaf $\fa \subset \O_X$ and a number $c \ge 0$, 
we define the \emph{Mather log discrepancy} of $E$ with respect to the pair $(X,\fa^c)$
to be
\[
\^a_E(X,\fa^c) := \ord_E(\jac_f)+1 - c\.\ord_E(\fa).
\]
When $X$ is smooth, we write $a_E(X)$ and $a_E(X,\fra^c)$ instead of $\^a_E(X)$ and $\^a_E(X,\fra^c)$, respectively.
It is  clear that the definition of Mather log discrepancy 
only depends on the valuation $\ord_E$ that $E$ defines on the
function field of $X$, and not on the model $Y$. 
We say that the pair $(X,\fa^c)$ is Mather log canonical if for every $E$ as above, we have
$\^a_E(X,\fa^c)\geq 0$. The \emph{Mather log canonical threshold} of the pair $(X,\fa)$, with $\fa$ a proper nonzero ideal of $R$,
is defined by
\[
\^\lct(\fa):= \sup \{\, c \in \R_{\ge 0} \mid \text{$(X,\fa^c)$ is Mather log canonical}\, \}.
\]
\end{defi}
It is straightforward to check that this is equivalent to the definition of $\^\lct(\fra)$ given in Introduction.
We put, by convention,  $\^\lct(0)=0$ and $\^\lct(\cO_X)=\infty$.

\begin{rmk}
We refer to \cite{Ish13} for basic facts about Mather log discrepancies and Mather log canonical threshold. 
A useful fact is that if $f\colon Y\to X$ is a log resolution of $(X,\fa)$ which factors through the Nash blow-up of $X$,
then there is a divisor $E$ on $Y$ such that $\^\lct(\fa)=\frac{\^a_E(X)}{\ord_E(\fa)}$.
\end{rmk}

We will use several times the following basic fact about divisorial valuations.

\begin{lem}
\label{l:div-val}
Let $f\colon X\to Y$ be a dominant morphism of varieties. If $E$ is a divisor over $X$, then the restriction of $\ord_E$ to 
$k(Y)$ is a multiple of a divisorial valuation, that is, we can write
$$\ord_E|_{k(Y)}=q\cdot\ord_F$$
for some divisor $F$ over $Y$ and some positive integer $q$.
\end{lem}

\begin{proof}
Let $v=\ord_E$ and $w=v|_{k(Y)}$. Note that $w$ is a valuation with center on $Y$, the center being the closure of the image
of the center of $v$ on $X$.
We denote by $R_v$ and $R_w$ the valuation rings corresponding to $v$ and $w$, respectively,
and by $k_v$ and $k_w$ the corresponding residue fields. 
Note that ${\rm trdeg}(k_w/k)\leq \dim(Y)$, with equality if and only if $w$
is the trivial valuation. Furthermore, $w$ is a multiple of a divisorial valuation if and only if
${\rm trdeg}(k_w/k)=\dim(Y)-1$ (see \cite[Lemma~2.45]{KM98}). On the other hand, since $v$ is a divisorial valuation,
we know that ${\rm trdeg}(k_v/k)=\dim(X)-1$. It follows from \cite[Chapter VI.6, Corollary~1]{ZS60} that
${\rm trdeg}(k_v/k_w)\leq{\rm trdeg}(k(X)/k(Y))=\dim(X)-\dim(Y)$. We conclude that ${\rm trdeg}(k_w/k)\geq\dim(Y)-1$.
Since it is clear that $w$ is not the trivial valuation, we conclude that in fact ${\rm trdeg}(k_w/k)=\dim(Y)-1$, hence
$w$ is a multiple of a divisorial valuation.
Since $w$ only takes integer values, it 
is immediate to see that the multiple is a positive integer.
\end{proof}

The next result gives an alternative way of computing Mather log discrepancies. Suppose that 
 $E$ is a prime divisor over a normal $n$-dimensional affine variety $X$.  
Given a closed immersion $X
\hookrightarrow \A^N$ 
and a general linear projection $\p \colon \A^N \to Y := \A^n$, we may write
$\ord_E|_{k(Y)} = q\.\ord_F$, for a prime divisor $F$ over $Y$ and a positive integer $q$,
by Lemma~\ref{l:div-val}.

\begin{prop}
\label{p:^k_E+1=q(k_F+1)}
With the above notation, we have
\[
\^a_E(X) = q\.a_F(Y).
\]
\end{prop}

\begin{proof}
Consider a commutative diagram
\[
\xymatrix{
X' \ar[d]_g \ar[r]^f & X \ar[d] \ar@{^(->}[r] & \A^N \ar[d]^\p \\
Y' \ar[r] & Y \ar@{=}[r] & \A^n 
}
\]
where $X' \to X$ and $Y' \to Y$ are resolutions such that $E$ is a divisor on $X'$
and $F$ is a divisor on $Y'$. Note that $\ord_E(g^*F) = q$ and $\ord_E(K_{X'/Y'}) = q-1$. 
Denoting by $h \colon X' \to Y$ the composition of $f$ with the projection to $Y$, 
we have $\ord_E(K_{X'/Y}) = \ord_E(\jac_h)$. 
If $x_1,\dots,x_n$ is a regular system of parameters in $X'$ centered at a general point of $E$
and $y_1,\dots,y_N$ are affine coordinates in $\A^N$, 
then $f$ is locally given by equations $y_i = f_i(x_1,\dots,x_n)$, and 
$\Jac_f$ is locally defined by the $n\times n$ minors of the matrix
$(\de f_i/\de x_j)$. For a linear projection $\p \colon \A^N \to Y = \A^n$, 
$\jac_h$ is locally defined by a linear combination of 
the $n\times n$ minors of $(\de f_i/\de x_j)$. 
If the projection is general, then so is the linear combination, and we have
$\^a_E(X) = \ord_E(K_{X'/Y})+1$.
Writing $K_{X'/Y} = K_{X'/Y'} + K_{Y'/Y}$, we get
\[
\^a_E(X) = \ord_E(K_{X'/Y'}) + \ord_E(g^*K_{Y'/Y}) + 1 = q\.a_F(Y).
\]
\end{proof}

The following theorem is a generalization of \cite[Theorem~1.1]{dFEM03} to Cohen--Macaulay varieties. 

\begin{thm}
\label{t:projection}
Let $X \subseteq \A^N$ be a Cohen--Macaulay variety of dimension $n$,
and $E$ a divisor over $X$.
For some $1 \le r \le n$, consider the morphism
\[
\f \colon X \to \A^{n-r+1} 
\]
induced by restriction of a general linear projection $\s \colon \A^N \to \A^{n-r+1}$.
Write $\ord_E|_{k(\A^{n-r+1})} = q \cdot \ord_G$, where $G$ is a prime divisor over $\A^{n-r+1}$ 
and $q$ is a positive integer (cf.\ Lemma~\ref{l:div-val}). 
Let $Z \hookrightarrow X$ a closed Cohen--Macaulay subscheme of pure codimension $r$
such that $\f|_Z$ is a finite morphism.
Note that $\f_*[Z]$ is a cycle of codimension one in $\A^{n-r+1}$;
we regard $\f_*[Z]$ as a Cartier divisor on $\A^{n-r+1}$. 
Then, for every $c\in\R_{\geq 0}$ such that $\^a_E(X,cZ)\geq 0$, we have
\begin{equation}
\label{t:eq:1}
q\. a_G\(\A^{n-r+1},\frac {r!\,c^r}{r^r}\.\f_*[Z]\) \le \^a_E(X,cZ).
\end{equation}
Moreover, if the ideal defining $Z$ in $X$ is locally generated by a regular sequence, then
\begin{equation}
\label{t:eq:2}
q\. a_G\(\A^{n-r+1},\frac {c^r}{r^r}\.\f_*[Z]\)  \le \^a_E(X,cZ).
\end{equation}
\end{thm}

\begin{proof}
Our argument is similar to the one used in the proof of \cite[Theorem 1.1]{dFEM03}.
We assume that $\ord_E(Z) > 0$ (the case $\ord_E(Z) = 0$ is easier and left to the reader).
We factor $\s$ as a composition of two general linear projections 
\[
\A^N \to U=\A^n \to V = \A^{n-r+1}.
\]
By Lemma~\ref{l:div-val}, we can write $\ord_E|_{k(U)} = p\.\ord_F$ for some prime divisor $F$ over $U$ and
some positive integer $p$. Note that $p$ divides $q$.

Let $h \colon V' \to V$ be a proper, birational morphism, with $V'$ smooth,  such that $G$ is a prime divisor on $V'$.
Let $X' := V' \times_VX$ and $U' := V' \times_VU$, and consider the 
induced commutative diagram with Cartezian squares
\[
\xymatrix{
X' \ar[d]^{\ff'} \ar[r]^f \ar@/_15pt/[dd]_{\f'} & X \ar[d]_\ff \ar@/^15pt/[dd]^\f\\
U' \ar[d]^{\g'} \ar[r]^{g} & U \ar[d]_\g \\
V' \ar[r]^h & V
}.
\]
Let $Z' := f^{-1}(Z) \hookrightarrow X'$ and $Z'' := \ff'(Z') \hookrightarrow U'$, 
both defined scheme-theoretically. 
In general, we have  $Z'' \hookrightarrow g^{-1}(\ff(Z))$, but this may be a proper subscheme.
First, note that $\psi$ is a finite, flat morphism. Finiteness follows from the fact that it is induced by a generic projection,
while flatness follows from the fact that it is finite, $U$ is smooth, and $X$ is Cohen-Macaulay. Since $\gamma$ is clearly
flat (in fact, smooth), we conclude that $\f$ is flat. Therefore both $X'$ and $U'$ are varieties and $f$ and $g$ are proper, birational morphisms.
Furthermore,
the restriction $\f'|_{Z'}$ is finite by base-change, and thus both $\ff'|_{Z'}$ and $\g'|_{Z''}$ are finite.

Note that $Z'$ is a closed subscheme of ${\psi'}^{-1}(Z'')$, hence
\begin{equation}
\label{eq1}
p\.\ord_F(Z'') = \ord_E((\ff')^{-1}(Z'')) \ge \ord_E(Z') = \ord_E(Z).
\end{equation}

Since $h$, being a morphism between two smooth varieties, 
is a locally complete intersection morphism, it follows by flat base change that 
$f$ is a locally complete intersection morphism as well.
More explicitly, $h$ factors as $h = h_1 \o h_2$
where $h_1 \colon V' \times V \to V$ is the projection and $h_2 \colon V' \inj V' \times V$
is the regular embedding given by the graph of $h$.   
By pulling back, we get a decomposition $f = f_1 \o f_2$
where $f_1 \colon V' \times X \to X$ is smooth and $f_2 \colon X' \inj V' \times X$
is a regular embedding of codimension equal to $\dim V=\dim V'$. 
Recall that the pull-back $f^*[Z]\in A_{n-r}(X')$ is defined as $f_2^![V' \times Z]$
(see \cite[Section 6.6]{Ful98}). 

We now show that $Z'$ is pure-dimensional, of the same dimension as $Z$, and $f^*[Z]$ is equal to the class of $[Z']$ in $A_{n-r}(X')$.
Since $\f'|_{Z'}$ is finite and $\f'(Z')$ is a proper subset of $V'$,
we see that $\dim Z' \le \dim V' - 1=n-r$. On the other hand, $Z'$ is locally cut out in 
$V' \times Z$ by $\dim V'$ equations, hence every irreducible component of $Z'$ 
has dimension at least $\dim Z=n-r$. Therefore $Z'$ is pure dimensional, of 
dimension $\dim Z$. Since $V'\times X$ is Cohen-Macaulay, it follows from \cite[Proposition 7.1]{Ful98}
that $f^*[Z]=[Z']$ in $A_{n-r}(X')$.

Since $\ff'|_{Z'} \colon Z' \to Z''$ is a finite, dominant  morphism of schemes, 
we see that $Z''$ is also pure dimensional of the same dimension as $Z'$, and $\ff'_*[Z'] \ge [Z'']$.
Note that $h^*\f_*[Z]$ and $\f'_*[Z']$ are divisors on $V'$. 
Since $f$ and $h$ are locally complete intersection morphisms of the same codimension, 
and since we have seen that $f^*[Z]=[Z']$ in $A_{n-r}(X')$, it follows from
\cite[Example~17.4.1]{Ful98} that $h^*\f_*[Z] \sim \f'_*[Z']$
(note that while $\f$ and $\f'$ are not proper morphisms, they are proper when restricted to the supports of
$Z$ and $Z'$, respectively). 
Since the two divisors are equal away from the 
exceptional locus of $h$, we deduce that $h^*\f_*[Z] = \f'_*[Z']$
by the Negativity Lemma (see~\cite[Lemma~3.39]{KM98}). We thus conclude that
\[
h^*\f_*[Z] =  \f'_*[Z'] \ge \g'_*[Z''].
\]
On the other hand, the center $C$ of $\ord_F$ in $U'$ is contained in $Z''$
and dominates $G$. Since $\f'|_{Z'}$ is finite, it follows that the map $\g'|_C \colon C \to G$ is finite. 
In particular, we have $\dim(C)=\dim(G)=n-r=\dim(Z'')$, hence 
$C$ is an irreducible component of $Z''$. 
Therefore we have
\begin{equation}
\label{eq2}
\ord_G(\f_*[Z]) = \ord_G(h^*\f_*[Z]) \ge \ord_G(\g'_*[Z'']) \ge e_C([Z'']) = \ell(\O_{Z'',C}).
\end{equation}

Let $b := k_G(V)$ denote the discrepancy of $G$ over $V$, 
and let $H := (\g')^*G$. Note that 
$p\.\val_F(H) = q$ and since $\gamma'$ is smooth, 
$H$ is a smooth divisor at the generic point of $C$. Moreover, since $K_{U'/U} = (\g')^*K_{V'/V}$, we have
$K_{U'/U} = b H + R$, where $R$ is a divisor that does not contain $C$ in its support.
Then, by Proposition~\ref{p:^k_E+1=q(k_F+1)} and equation \eqref{eq1}, we see that
\[
\^a_E(X,cZ) \ge p\. a_F(U',cZ'' - K_{U'/U}) = p\.a_F(U',cZ'' - b H).
\]
Setting $\a := \^a_E(X,cZ)/q$, we have 
$$
a_F(U', cZ'' - (b- \a)H)=a_F(U',cZ''-bH)-\a\cdot\ord_F(H)\leq \a(1-\ord_F(H)) \le 0,
$$
where the last inequality follows from the fact that $\ord_F(H)\geq 1$ and, by assumption, $\a\geq 0$. 
This in turn implies 
\begin{equation}
\label{eq3}
\ell(\O_{Z'',C}) \ge \frac{(b -\a + 1)r^r}{r!\,c^r}.
\end{equation}
Indeed, if $b - \a \ge 0$, then
\eqref{eq3} follows by \cite[Theorem~2.1]{dFEM03}. The case $b- \a < 0$ is easier, 
and follows from \cite[Lemma~2.4]{dFEM03} using the same degeneration to monomial ideals
(see \cite[Section~2]{dFEM04}). 

Combining \eqref{eq2} and \eqref{eq3}, we get
\[
q\. a_G\(V,\frac {r!\,c^r}{r^r}\.\f_*[Z]\)=q\cdot a_G(V)-\frac {r!\,c^r}{r^r}\cdot\ord_G(\f_*[Z]) \le q(b+ 1 - (b -\a + 1)) = \^a_E(X,cZ),
\]
as stated in \eqref{t:eq:1}

Suppose now that the ideal of $Z$ in $X$ is locally generated by a regular sequence. 
If $I_Z \subseteq \O_X$ is the 
ideal sheaf of $Z$ and $Z_i$ is an irreducible component of $Z$, then 
\begin{equation}
\label{eq:4}
\ell(\O_{Z,Z_i}) = e(I_Z\O_{X,Z_i}) = \lim_{m \to \infty} 
\frac{r!}{m^r}\.\ell(\O_{X,Z_i}/I_Z^m\O_{X,Z_i}).
\end{equation}

For every $m$, let $Z_m \hookrightarrow X$ be the subscheme defined by $I_Z^m$. 
Since $I_Z$ is locally generated by a regular sequence, $I_Z^{m}/I_Z^{m+1}$ is a locally free $\O_Z$-module, 
and thus is Cohen--Macaulay (as an $\O_Z$-module, hence as an $\O_X$-module). Note that
$\O_Z$ is also Cohen--Macaulay (as an $\O_Z$-module, hence as an $\O_X$-module).
By applying \cite[Proposition~1.2.9]{BH93} to the exact sequences of $\O_X$-modules
\[
0 \to I_Z^{m}/I_Z^{m+1} \to \O_{Z_{m+1}} \to \O_{Z_m} \to 0,
\]
we see by induction that $\O_{Z_m}$ is a Cohen--Macaulay $\O_X$-module, and therefore
$Z_m$ is a Cohen--Macaulay scheme. Note that
\[
\^a_E\(X,\frac cm\. Z_m\) = \^a_E(X,cZ) \quad\text{for all}\,\, m, 
\]
and 
\[
\lim_{m\to\infty} \frac{r!}{m^r}\.[Z_m] = [Z]
\]
by \eqref{eq:4}. 
Since $\f|_{Z_m}$ if finite for every $m$, we may apply \eqref{t:eq:1} 
with $(Z,c)$ replaced by $(Z_m,c/m)$ to deduce, after letting
$m$ go to infinity, the inequality in 
\eqref{t:eq:2}.
\end{proof}

\begin{cor}
\label{c:projection}
With the same assumptions as in the first part of Theorem~\ref{t:projection}, we have
\begin{equation}
\label{eq:proj1}
\lct(\A^{n-r+1},\f_*[Z]) \le \frac{\^\lct(X,Z)^r}{r^r/r!}. 
\end{equation}
Moreover, if the ideal of $Z$ in $X$ is locally generated by a regular sequence, then 
\begin{equation}
\label{eq:proj2}
\lct(\A^{n-r+1},\f_*[Z]) \le \frac{\^\lct(X,Z)^r}{r^r}. 
\end{equation}
\end{cor}

\begin{proof}
We apply Theorem~\ref{t:projection} for a divisor $E$ computing $\^\lct(X,Z)$. 
\end{proof}

We apply the first part of the corollary to prove Theorem~\ref{t:e-Mather-lct}. 

\begin{proof}[Proof of Theorem~\ref{t:e-Mather-lct}]
Let $x\in X$ be the cosupport of $\fa$.
After replacing $X$ by an affine neighborhood of $x$, we may assume that 
we have a closed immersion $X\hookrightarrow\A^N$.
Let $m\geq 1$ be fixed and
$Z_m \hookrightarrow X$ be the zero-dimensional scheme defined by $\fa^m$.
Note that $Z_m$ is Cohen--Macaulay, since it is zero dimensional.
 
Consider a general linear projection $\A^N \to \A^1$ and let $\f \colon X \to \A^1$
be the induced map. Note that 
\[
\^\lct(X,Z_m) = \frac 1m \.\^\lct(X,Z), 
\]
and since
\[
\f_*[Z_m] = \ell(\O_X/\fa^m)\.[f(x)], 
\]
we have 
\[
\lct(\A^1,\f_*[Z_m]) = \frac 1{\ell(\O_X/\fa^m)}.
\]
Then \eqref{eq:proj1} gives
\[
\frac{\ell(\O_X/\fa^m)}{m^n/n!} \. \^\lct(X,Z)^n \ge n^n.
\]
Setting $m = 1$ and taking $n$-th roots, we get \eqref{eq:ell-Mather-lct}.
The formula \eqref{eq:e-Mather-lct} follows by taking the limit as $m$ goes to infinity
and then taking $n$-th roots.
\end{proof}

\section{The volume of a graded sequence of ideals}

We recall, following \cite{ELS03} and \cite{Mus02a}, some basic facts about the volume of a graded sequence of ideals.
Let $k$ be an algebraically closed field of arbitrary characteristic and let $X=\Spec(R)$ be an $n$-dimensional affine variety over $k$
(in particular, we assume that $R$ is a domain). 
Recall that a sequence $\fa_\bullet = (\fa_m)_{m \ge 0}$ of ideals $\fa_m \subseteq R$ is
a graded sequence of ideals if $\fa_0 = R$ and $\fa_p\.\fa_q \subseteq \fa_{p+q}$ for every $p,q \ge 1$.

\begin{defi}
The \emph{volume} of a graded sequence $\fa_{\bullet}$ is defined by
\[
\vol(\fa_{\bullet}):=\limsup_{m\to\infty}\frac{\ell(R/\fa_m)}{m^n/n!}.
\]
\end{defi}


Let $\fra_{\bullet}$ be a graded sequence of ideals in $R$. The main case for understanding the notion of volume 
is that when there is a closed point $x$ in $X$ such that for every $m\geq 1$, the cosupport of $\fra_m$ is equal to $\{x\}$ (we say that $\fa_{\bullet}$ is \emph{cosupported at} $x$).
Note that in this case we have $\vol(\fa_{\bullet})<\infty$. Indeed, if $N$ is a positive integer such that $\fm_x^N\subseteq\fa_1$, where $\fm_x$ is the ideal defining $x$,
then $\fm_x^{pN}\subseteq\fa_p$ for every $p\geq 1$, hence $\vol(\fa_{\bullet})\leq N^n\cdot e(\fm_x)$. In fact, under the same assumption,
it follows from 
\cite[Theorem~3.8]{LM09} that the volume of $\fra_{\bullet}$ can be computed 
as a limit of normalized Hilbert-Samuel multiplicities. More precisely, we have
\begin{equation}
\label{eq:vol=lim-e}
\vol(\fa_{\bullet}) = \lim_{m\to\infty}\frac{e(\fa_m)}{m^n}.
\end{equation}
Moreover, the limit superior in the definition of volume is a limit
\begin{equation}\label{eq_limit_in_volume}
\vol(\fa)=\lim_{m\to\infty}\frac{\ell(R/\fa_m)}{m^n/n!}
\end{equation}
by \cite[Theorem~1]{Cut14}. 

\begin{rmk}\label{vol_as_inf_e}
Suppose that $\fra_{\bullet}$ is a graded sequence of ideals such that $\fra_p\subseteq\fra_q$ whenever $p\geq q$.
If $\fra_{\bullet}$ is cosupported at a point $x\in X$, then
$$\vol(\fra_{\bullet})=\inf_{m\geq 1}\frac{e(\fa_m)}{m^n}.$$
Indeed, this is a consequence of (\ref{eq:vol=lim-e}) and of the fact that
$$\lim_{m\to\infty}\frac{e(\fa_m)}{m^n}=\inf_{m\geq 1}\frac{e(\fa_m)}{m^n}.$$
This equality is a consequence of Lemma~\ref{lem2} below.
\end{rmk}

\begin{rmk}\label{ideals_cosupported_finite set}
Suppose that $\fa_{\bullet}$ is a graded sequence of ideals and $\Gamma=\{x_1,\ldots,x_r\}$ is a finite set of
closed points in $X$ such that for every $m\geq 1$, the ideal $\fa_m$ has cosupport $\Gamma$. For every $m\geq 1$,
let us consider the primary decomposition 
$$\fa_m=\bigcap_{i=1}^r\fa_m^{(i)},$$ where each $\fa_m^{(i)}$ is an ideal with cosupport $\{x_i\}$. It is clear that each $\fa^{(i)}_{\bullet}$
is a graded sequence of ideals. Since
\begin{equation}\label{eq1_ideals_cosupported_finite set}
\ell(R/\fa_m)=\sum_{i=1}^r\ell(R/\fa_m^{(i)}),
\end{equation}
we deduce 
\begin{equation}\label{eq2_ideals_cosupported_finite set}
\vol(\fa_{\bullet})=\sum_{i=1}^r\vol(\fa_{\bullet}^{(i)}).
\end{equation}
 In particular, we see that $\vol(\fra_{\bullet})<\infty$ and the assertion in 
(\ref{eq_limit_in_volume}) also holds for $\fra_{\bullet}$.
\end{rmk}

\begin{eg}\label{volume_integral_closure}
Suppose that $\fra_{\bullet}$ is a graded sequence of ideals such that each $\fra_m$, with $m\geq 1$, has cosupport equal to a finite set $\Gamma$.
If $\overline{\fra}_{\bullet}$ is such that $\overline{\fra}_m$ is the integral closure of the ideal $\fa_m$, then $\overline{\fra}_{\bullet}$ is a graded sequence
and $\vol(\overline{\fra}_{\bullet})=\vol(\fra_{\bullet})$. The first assertion follows from the fact that $\overline{\fra}_p\cdot\overline{\fra}_q$
is contained in the integral closure of $\fra_p\cdot\fra_q$, hence in $\overline{\fra}_{p+q}$. In order to see that $\vol(\fra_{\bullet})=\vol(\overline{\fra}_{\bullet})$,
we may assume that all $\fra_m$ have cosupport at the same point $x\in X$ (see Remark~\ref{ideals_cosupported_finite set}). In this case, since
$e(\fra_m)=e(\overline{\fra}_m)$ for every $m$, the assertion follows from (\ref{eq:vol=lim-e}).
\end{eg}

Under a mild condition on $\fra_{\bullet}$
which is often satisfied, 
we give in the next proposition a new easy proof of the assertions (\ref{eq:vol=lim-e}) and (\ref{eq_limit_in_volume})
 in the smooth case.

\begin{prop}\label{prop1}
Suppose that $X=\Spec(R)$ is smooth. If
$\fa_{\bullet}$ is a graded sequence of ideals in $R$ which is cosupported at a point in $X$,
and $\fa_p\subseteq\fa_q$ whenever $p\geq q$, then
\begin{equation}\label{eq1_prop1}
\vol(\fa_{\bullet}) = \lim_{m\to\infty}\frac{\ell(R/\fa_m)}{m^n/n!}
=\inf_{m\geq 1}\frac{\ell(R/\fa_m)}{m^n/n!}
\end{equation}
\begin{equation}\label{eq2_prop1}
=\lim_{m\to\infty}\frac{e(\fa_m)}{m^n}
=\inf_{m\geq1}\frac{e(\fa_m)}{m^n}.
\end{equation}
\end{prop}

Note that while the proposition recovers (\ref{eq:vol=lim-e}) and (\ref{eq_limit_in_volume}) in the smooth setting,
it also implies the equality $\vol(\fra_{\bullet})=\inf_{m\geq 1}\frac{\ell(R/\fa_m)}{m^n/n!}$, which needs the
smoothness assumption.
For the proof of the proposition we need two lemmas.
The first one is a special case of \cite[Lemma~25]{KN};  
for completeness, we include the proof of this special case.

\begin{lem}\label{lem1}
If $X=\Spec(R)$ is smooth, 
$x\in X$ is a closed point defined by $\fm_x$, and $\fa$ is an $\fm_x$-primary ideal in $R$, 
then for every $p\geq 1$, we have
\[
\ell(R/\fa)\geq\frac{1}{p^n}\.\ell(R/\fa^p).
\]
\end{lem}

\begin{proof}
Since $X$ is smooth, it is straightforward to reduce to the case when 
$X=\A^n$ and $\fra$ is an ideal supported at the origin.
We choose a monomial order on $R=k[x_1,\ldots,x_n]$ and for every ideal 
$\fb$ in $R$, we consider the initial ideal
\[
\ini(\fb)=(\ini(f)\mid f\in\fb).
\]
We refer to \cite[Chapter 15]{Eis95} for the basic facts about initial ideals.
Note that we have $\ell(R/\fb)=\ell(R/\ini(\fb))$.
It is clear that $\ini(\fa^p)\supseteq \ini(\fa)^p$. 
It follows that if we know the assertion in the lemma for 
$\ini(\fa)$, then
\[
\ell(R/\fa)=\ell(R/\ini(\fa))\geq \frac{1}{p^n}\.\ell(R/\ini(\fa)^p)
\geq \frac{1}{p^n}\.\ell(R/\ini(\fa^p))=\frac{1}{p^n}\.\ell(R/\fa^p),
\]
hence we obtain the assertion for $\fa$. 

The above argument shows that we may assume that $\fa$ is a monomial ideal. 
For every such ideal $\fa$, we consider the sets 
\[
Q(\fa):=\bigcup_{x^u\in\fa}(u+\R^n) \quad\text{and}\quad 
Q^c(\fa):=\R^n_{\geq 0}\smallsetminus Q(\fa).
\]
Note that $Q^c(\fra)$ is equal, up to a set of measure zero,
to the union of $\ell(R/\fa)$ disjoint open unit cubes. 
Therefore $\ell(R/\fa)$ is equal to $\vol(Q^c(\fra))$, the Euclidean volume of
$Q^c(\fa)$. On the other hand, it is clear from definition that 
$Q(\fa^p)\supseteq p\cdot Q(\fa)$, hence $Q^c(\fa^p)\subseteq p\cdot Q^c(\fa)$.
We thus conclude 
\[
\ell(R/\fa)=\vol(Q^c(\fa))\geq\vol\left(\frac{1}{p}\.Q^c(\fa^p)\right)
=\frac{1}{p^n}\.\vol(Q^c(\fa^p))=\frac{1}{p^n}\.\ell(R/\fa^p).
\]
This completes the proof of the lemma.
\end{proof}

The following is a variant of \cite[Lemma~2.2]{Mus02a}. 

\begin{lem}\label{lem2}
If $(\alpha_m)_{m\geq 1}$ is a sequence of non-negative real numbers that satisfies the following two conditions:
\begin{enumerate}
\item[i)] $\alpha_{pq}\leq p\cdot\alpha_q$ for every $p,q\geq 1$, and
\item[ii)] $\alpha_p\geq\alpha_q$ whenever $p\geq q$,
\end{enumerate}
then
\[
\lim_{m\to\infty}\frac{\alpha_m}{m}=\inf_{m\geq 1}\frac{\alpha_m}{m}.
\]
\end{lem}

\begin{proof}
Let $\lambda:=\inf_m\frac{\alpha_m}{m}$. We need to show that for every $\epsilon>0$, we have
$\frac{\alpha_m}{m}\leq\lambda+\epsilon$ for all $m\gg 1$. By definition, there is $d>0$ such that
$\frac{\alpha_d}{d}<\lambda+\frac{\epsilon}{2}$. Given $m$, we write $m=jd-i$, 
where $0\leq i<d$ (hence $j=\lceil m/d\rceil$).
The hypotheses imply
\[
\frac{\alpha_m}{m}\leq\frac{\alpha_{jd}}{jd-i}\leq \frac{\alpha_{d}}{d}\cdot\frac{jd}{jd-i}\leq \left(\lambda+\frac{\epsilon}{2}\right)\cdot \frac{jd}{jd-i}.
\]
For $m\gg 1$, we have $j\gg 1$, hence 
$\frac{jd}{jd-i}<\frac{\lambda+\epsilon}{\lambda+\frac{\epsilon}{2}}$.
This completes the proof of the lemma.
\end{proof}

\begin{proof}[Proof of Proposition~\ref{prop1}]
Let $\alpha_m=\ell(R/\fra_m)$. If $p\geq q$, then by assumption 
$\fra_p\subseteq\fra_q$, hence $\alpha_p\geq\alpha_q$. 
Moreover, it follows from Lemma~\ref{lem1} that $\alpha_{pq}\leq p\cdot\alpha_q$ for all 
$p,q\geq 1$. The two equalities in (\ref{eq1_prop1}) now follow from the definition of volume and
Lemma~\ref{lem2}.

Note now that Lemma~\ref{lem2} also gives the second equality in (\ref{eq2_prop1}). Indeed, for $p\geq q$,
we have $\fa_p\subseteq\fa_q$, hence $e(\fa_p)\geq e(\fa_q)$; moreover, the inclusion $\fa_q^p\subseteq \fa_{pq}$
implies $e(\fa_{pq})\leq e(\fa_q^p)=p^n\cdot e(\fa_q)$. In order to prove the first equality in (\ref{eq2_prop1}), note first that
by definition of Hilbert--Samuel multiplicity, for every $m$ we have
$$e(\fa_m)=\lim_{q\to\infty}\frac{\ell(R/\fa^q_m)}{q^n/n!},$$
hence using Lemma~\ref{lem1} we conclude that $e(\fa_m)\leq\frac{\ell(R/\fa_m)}{n!}$.
Dividing by $m^n$ and passing to limit, we obtain
$$L:=\lim_{m\to\infty}\frac{e(\fa_m)}{m^n}\leq \vol(\fa_{\bullet}).$$
In order to prove the reverse inequality, note that given any $\epsilon>0$, by definition of $L$ and of the Hilbert--Samuel multiplicity,
we can find first $m\geq 1$ and then $q\geq 1$ such that $L>\frac{\ell(R/\fa_m^q)}{m^nq^n/n!}-\epsilon$.
Since $\fa_m^q\subseteq\fa_{mq}$, it follows that 
$$L>\frac{\ell(R/\fa_{mq})}{(mq)^n/n!}-\epsilon\geq\inf_p\frac{\ell(R/\fa_p)}{p^n/n!}-\epsilon.$$
Since this holds for every $\epsilon>0$, using (\ref{eq1_prop1}) we conclude that $L\geq\vol(\fa_{\bullet})$, completing the proof of the proposition.
\end{proof}

\begin{rmk}
Suppose that $X=\Spec(R)$ is smooth and $\fa$ is an ideal in $R$ which is cosupported at a point.
Applying Proposition~\ref{prop1} in the case of the sequence given by the powers of $\fa$,
we see that 
\[
e(\fa)=\inf_{m\geq 1}\frac{\ell(R/\fa^m)}{m^n/n!}.
\]
\end{rmk}

In this note we will be interested in graded sequences that arise from pseudo-valuations.

\begin{defi}
A function $v\colon R\to\R_{\geq 0}\cup\{\infty\}$ is said to be
a \emph{pseudo-valuation} of $R$ if it satisfies the following conditions:
\begin{equation}\tag{i}\label{eq0_v}
v(0)=\infty \quad\text{and}\quad v(\lambda)=0\quad\text{for every}\,\,\lambda\in k,
\end{equation}
\begin{equation}\tag{ii}\label{eq1_v}
v(f+g)\geq\min\{v(f),v(g)\}\quad \text{for every}\,\,f,g\in R,\,\,\text{and}
\end{equation}
\begin{equation}\tag{iii}\label{eq2_v}
v(fg)\geq v(f)+v(g)\quad \text{for every}\,\,f,g\in R.
\end{equation}
We say that a pseudo-valuation $v$ is \emph{radical} if, in addition, it satisfies
\begin{equation}\tag{iv}\label{eq3_v}
v(f^r)=r\cdot v(f)\quad \text{for every}\,\,f\in R, r\in {\mathbb Z}_{>0}.
\end{equation}

The \emph{support} of a pseudo-valuation $v$ is the closed subscheme $\Supp(v) \hookrightarrow X$ 
defined by the ideal 
\[
\fb_\infty(v) := \{f \in R \mid v(f) = \infty\}.
\]
Given a pseudo-valuation $v$ and an ideal $\fa$ in $R$, we put 
\[
v(\fa):=\inf\{v(f)\mid f\in\fa\}.
\]
We say that $v$ \emph{has center at} the closed subscheme $Y$ defined by the ideal $\fb$ in $R$ if
we have $\fb=\{f\in R\mid v(f)>0\}$. 
\end{defi}

\begin{rmk}\label{value_at_center}
Note that if $\fb$ defines the center of $v$, then $v(\fb)>0$. Indeed, if we put $I_m:=\{f\in R\mid v(f)\geq 1/m\}$, then each $I_m$ is an ideal in $R$
and we have $I_m\subseteq I_{m+1}$. Since $\fb=\bigcup_mI_m$ and $R$ is Noetherian, it follows that $\fb=I_m$ for $m\gg 0$. 
\end{rmk}

\begin{rmk}
There are two other related notions. A \emph{semi-valuation} of $R$ is a pseudo-valuation with the property that the inequality in (iii) is an equality for all $f$ and $g$
(in this case, condition (iv) is automatically satisfied). A semi-valuation $v$ is a \emph{valuation} if, in addition, we have $v(f)<\infty$ for all $f\in R\smallsetminus\{0\}$. 
It is clear that in this case we can extend $v$ to a valuation of the function field of $R$ by putting $v(f/g)=v(f)-v(g)$ for every nonzero $f,g\in R$. 
Note that if $v$ is a semi-valuation, then the ideal $\fb_{\infty}(v)$ is a prime ideal and we have 
a valuation $\overline{v}$ on $R/\fb_{\infty}$ such that $v=\overline{v}\circ \pi$, where $\pi\colon R\to R/\fb_{\infty}$ is the canonical projection. 
\end{rmk}

\begin{rmk}
If $(v_{\alpha})_{\alpha\in\Lambda}$ is a family of semi-valuations of $R$ and we put $v(f):=\inf_{\alpha\in\Lambda}v_{\alpha}(f)$, then $v$ is a radical pseudo-valuation. 
Note that the support of $v$ is the union of the supports of the $v_{\alpha}$ and if $\Lambda$ is finite, then the center of $v$ is the union of the centers of the $v_{\alpha}$.
In particular, these sets are not necessarily irreducible. It is a theorem of Bergman that every radical pseudo-valuation arises in this way. More precisely, for every
radical pseudo-valuation $w$ of $R$, there is a family $(w_i)_{i\in I}$ of semi-valuations of $R$ such that $w(f)=\inf_iw_i(f)$ for every $f\in R$ 
(see \cite[Theorem 2]{Ber71}).
\end{rmk}

\begin{rmk}
There is a canonical way to obtain a radical pseudo-valuation of $R$ from an arbitrary pseudo-valuation.
Indeed, if $v$ is any pseudo-valuation, then we put
$$\widetilde{v}(f):=\inf_{m\geq 1}\frac{v(f^m)}{m}=\lim_{m\to\infty}\frac{v(f^m)}{m},$$
where the second equality follows from property (iii) and a version of Lemma~\ref{lem2}
(see \cite[Lemma 1.4]{Mus02a}). It is easy to see that $\widetilde{v}$ is a radical pseudo-valuation
such that $\widetilde{v}(f)\leq v(f)$ for every $f\in R$. Moreover, if $w$ is another radical pseudo-valuation such that
$w(f)\leq v(f)$ for every $f\in R$, then $w(f)\leq \widetilde{w}(f)$ for every $f\in R$.
\end{rmk}

Suppose that $v$ is a pseudo-valuation of $R$. 
We define for every $m\in\Z_{\geq 0}$
\[
\fb_m(v):=\{f\in R \mid v(f)\geq m\}.
\]
It follows from \eqref{eq1_v} and \eqref{eq2_v} that
$\fb_{\bullet}(v) = (\fb_m(v))_{m \ge 0}$ is a graded sequence of ideals.

\begin{rmk} 
The sequence $\fb_{\bullet}(v)$ clearly satisfies
the condition $\fb_p(v)\subseteq\fb_q(v)$ for $p\geq q$.
\end{rmk}

\begin{eg}\label{valuation_from_ideal}
Suppose that $I\neq R$ is an ideal of $R$. If for every $f\in R$, we put
$v_I(f):=\min\{m\geq 0\mid f\in I^m\}$, then $v_I$ is a pseudo-valuation of $R$, with support $X$ and whose center is defined by $I$.
It follows from definition that in this case $\fb_m(v_I)=I^m$.
\end{eg}

\begin{rmk}\label{radical_of_ideal_b}
It is clear that for every pseudo-valuation $v$ and every $m\geq 1$, if $\fb$ is the ideal defining the center of $v$, then
$\fb_m(v)\subseteq \fb$ and the two ideals have the same radical. In fact, if $d$ is an integer
such that $d\cdot v(\fb)\geq 1$ (see Remark~\ref{value_at_center}), then
$\fb^{dm}\subseteq\fb_m(v)$ for every $m\geq 1$.
\end{rmk}

We will be mostly interested in pseudo-valuations with $0$-dimensional center.

\begin{defi}
The \emph{volume} of a pseudo-valuation $v$ of $R$
is defined to be the volume
\[
\vol(v) := \vol(\fb_\bullet(v))
\]
of the graded sequence $\fb_{\bullet}(v)$. 
Recall that by (\ref{eq:vol=lim-e}) and (\ref{eq_limit_in_volume}), we have
\[
\vol(v) =\lim_{m\to\infty}\frac{\ell(R/\fb_m(v))}{m^n/n!}=\lim_{m\to\infty}\frac{e(\fb_m(v))}{m^n}.
\]
\end{defi}

\begin{rmk}
\label{r:A}
We have $\vol(v) < \infty$ if and only if the center of $v$ is a finite set. Indeed, if the latter condition holds, then  
the finiteness of the volume follows from Remark~\ref{ideals_cosupported_finite set}. On the other hand,
if the center of $v$ has positive dimension, then $\ell(R/\fb_m(v))=\infty$ for all $m\geq 1$ by Remark~\ref{radical_of_ideal_b}.
\end{rmk}

\begin{eg}\label{eg_multiplicity_ideal}
If $I\neq R$ is an ideal whose cosupport consists of one point and $v_I$ is the pseudo-valuation associated to $I$ in Remark~\ref{valuation_from_ideal},
then $\vol(v_I)=e(I)$. 
\end{eg}

\begin{eg}\label{Rees_valuations}
Let $I\neq R$ be an ideal in $R$. Recall that there are finitely many divisorial valuations $w_1,\ldots,w_r$ of $R$ (the \emph{Rees valuations} of $I$)
with the property that for every $m\geq 0$, the integral closure $\overline{I^m}$ of $I^m$ is equal to
$$\{f\in R\mid w_i(f)\geq m\cdot w_i(I)\,\,\text{for}\,\,1\leq i\leq r\}.$$
We refer to \cite{Swa11} for an introduction to Rees valuations. 
In particular, we see that if $w$ is the pseudo-valuation given by
$w=\min_i\frac{1}{w_i(I)}w_i$, then $b_m(w)=\overline{I^m}$ for every $m$.
In particular, it follows from Example~\ref{volume_integral_closure} that if the cosupport of $I$ consists of one point, 
then $\vol(w)=e(I)$.
\end{eg}

\begin{eg}\label{eg_ineq_valuations}
Suppose that $v$ and $w$ are pseudo-valuations of $R$ such that $v(f)\geq w(f)$ for all $f\in R$.
In this case we have $\fb_m(w)\subseteq\fb_m(v)$ for all $m$. By taking the colength, dividing by $m^n/n!$, and
passing to limit, we obtain $\vol(w)\geq \vol(v)$.
\end{eg}

\begin{eg}\label{eg_scalar_multiple}
If $v$ is a pseudo-valuation of $R$ and $\alpha$ is a positive real number, then $\alpha v$ is a pseudo-valuation such that
$\vol(\alpha v)=\frac{1}{\alpha^n}\cdot\vol(v)$. Indeed, note that we have
$$\fb_m(\alpha v)\supseteq \fb_{\lceil m/\alpha\rceil}(v),$$
hence
$$\vol(\alpha v)\leq \lim_{m\to\infty}\frac{\ell(R/\fb_{\lceil m/\alpha\rceil}(v))}{\lceil m/\alpha\rceil^n/n!}\cdot \frac{\lceil m/\alpha\rceil^n}{m^n}=\vol(v)\cdot \frac{1}{\alpha^n}.$$
By writing $v=\frac{1}{\alpha}(\alpha v)$ and applying the inequality already proved, we obtain
$\vol(v)\leq\alpha^n\cdot\vol(\alpha v)$. By combining the two inequalities, we obtain $\vol(\alpha v)=\frac{1}{\alpha^n}\cdot\vol(v)$.
\end{eg}

The following proposition gives an important example of valuation with positive volume.

\begin{prop}
\label{p:ELS}
If $v$ is a divisorial valuation of $R$ having center at a closed point $x\in X$ and $X$ is analytically unramified\footnote{This means that the completion
$\widehat{\cO_{X,x}}$ is a domain (note that it is always reduced, since $\cO_{X,x}$ is a reduced excellent ring). The condition is satisfied, for example,
if $X$ is normal.} at $x$, then $\val(v) > 0$. 
\end{prop}

\begin{proof}
This is an immediate consequence of Izumi's theorem (see for example \cite[Theorem~1.2]{HS01}). 
This says that since the local ring $\cO_{X,x}$ is analytically unramified, there is a constant $c=c(v)$ such that
for every other divisorial valuation $v'$ with center $\{x\}$, we have
$v(f)\leq c\cdot v'(f)$ for every $f\in R$. 
Let $w_1,\ldots,w_r$ be the Rees valuations corresponding to the maximal ideal $\fm_x$ defining $x$. If $w=\min_i\frac{1}{w_i(\fm_x)}w_i$ and $\alpha=c\cdot \max_iw_i(\fm_x)$, 
then we see that $v(f)\leq \alpha\cdot w(f)$ for every $f\in R$. By combining Examples~\ref{Rees_valuations}, ~\ref{eg_ineq_valuations}, and \ref{eg_scalar_multiple}, 
we conclude that
$$\vol(v)\geq\vol(\alpha\cdot w)=\frac{\vol(w)}{\alpha^n}=\frac{e(\fm_x)}{\alpha^n}>0.$$
\end{proof}

\section{The volume of a subset in the space of arcs}

Suppose, as in the previous section, that $X=\Spec(R)$ is an $n$-dimensional,
affine algebraic variety over 
an algebraically closed field $k$. We now assume that ${\rm char}(k)=0$.

Let $X_{\infty}$ be the scheme of arcs of $X$ (for an introduction to spaces of arcs, see for example \cite{EM09}).
Since $X$ is affine, $X_{\infty}$ is affine as well, but in general not of finite type over $k$.
Note that if $\gamma\in X_{\infty}$ is a point with residue field $k(\gamma)$, then we can identify $\gamma$ with 
a morphism $\Spec(k(\gamma)\llbracket t\rrbracket)\to X$. We denote by $\pi\colon X_{\infty}\to X$ the canonical projection
taking $\gamma$ to $\gamma(0)$, the image by $\gamma$ of the closed point.

\begin{rmk}
While $X_{\infty}$ is not a Noetherian scheme, if $C$ is a closed subset of $X_{\infty}$, we may still consider the
irreducible components of $C$: these correspond to the prime ideals in $\cO(X_{\infty})$ which are minimal over the ideal
of $C$. Note that we can still write $C$ as the union of its irreducible components: this is an immediate application of Zorn's Lemma.
\end{rmk}

For every $\gamma\in X_{\infty}$, we define the function
$\ord_{\gamma}\colon R\to \Z_{\geq 0}\cup\{\infty\}$ given by 
$\ord_{\gamma}(f)=\ord_t(\gamma^*(f))$. It is clear that $\ord_{\gamma}$ is a semi-valuation of $R$.

Given a subset $C\subseteq X_{\infty}$, 
we consider the function $\ord_C\colon R \to\Z_{\geq 0}\cup\{\infty\}$
defined by 
\[
\ord_C(f)=\min_{\gamma\in C}\ord_{\gamma}(f).
\]
It follows from the definition that $\ord_C$ is a radical pseudo-valuation. 
For short, we denote $\fb_m(C) := \fb_m(\ord_C) \subseteq R$
and, similarly, let $\fb_\bullet(C) := \fb_\bullet(\ord_C)$.

\begin{lem}
\label{l:ov-C}
If $\ov C$ is the closure of a subset $C\subseteq X_{\infty}$, then $\ord_{\overline{C}} = \ord_C$.
\end{lem}

\begin{proof}
The assertion follows from the fact that for every $f\in R$ and every $m\in\Z$, the set $\{\gamma\in X_{\infty}\mid \ord_{\gamma}(f)\geq m\}$
is closed. 
\end{proof}

The assertion in the next lemma follows directly from definition. 

\begin{lem}\label{valuation_union_subsets}
If $C=\bigcup_{i\in I}C_i$, then $\ord_C(f)=\min_{i\in I}\ord_{C_i}(f)$ for every $f\in R$.
\end{lem}

\begin{rmk}
If $C$ is irreducible, then $\ord_C$ is a semi-valuation. Indeed, it follows from Lemma~\ref{l:ov-C} that if $\delta$ is the generic point
of $\ov C$, then $\ord_C=\ord_{\delta}$, hence $\ord_C$ is a semi-valuation.
\end{rmk}

\begin{rmk}
\label{r:B}
The center of the pseudo-valuation $\ord_C$ is equal to $\overline{\pi(C)}$, with the reduced scheme structure. Indeed, this follows from the fact that for $f\in R$ and $\gamma\in X_{\infty}$,
we have $\ord_{\gamma}(f)\geq 1$ if and only if $f$ lies in the ideal defining $\overline{\pi(\gamma)}$.
\end{rmk}

\begin{defi}
We define the \emph{volume} $\vol(C)$ of a set $C \subseteq X_\infty$ to be the volume
\[
\vol(C) := \vol(\ord_C) = \vol(\fb_\bullet(C))
\] 
of the pseudo-valuation $\ord_C$. 
\end{defi}

\begin{prop}
\label{p:vol=infty}
For every $C \subseteq X_\infty$, we have 
$\vol(C) < \infty$ if and only if $\pi(C)$ is a finite set of closed points.
\end{prop}

\begin{proof}
The assertion follows by combining Remarks~\ref{r:A} and~\ref{r:B}.
\end{proof}

From now on, we restrict our attention to subsets $C\subseteq X_{\infty}$ whose image in $X$ is a finite set of closed points.
In the next propositions, we give some basic properties of volumes of subsets of $X_{\infty}$.

\begin{prop}
\label{p:C1-C2}
If $C_1 \subseteq C_2$, then
$\vol(C_1) \le \vol(C_2)$.
\end{prop}

\begin{proof}
If $C_1\subseteq C_2$ then it is clear that $\ord_{C_1}(f)\geq\ord_{C_2}(f)$ for every $f\in R$.
The assertion then follows from Example~\ref{eg_ineq_valuations}.
\end{proof}

The next proposition allows us to reduce to considering subsets lying in a fiber of $\pi\colon X_{\infty}\to X$.
For every closed point $x\in X$, we denote the fiber $\pi^{-1}(x)$ by $X_{\infty}(x)$.

\begin{prop}\label{p:decomposition}
Let $C\subseteq X_{\infty}$ be such that $\pi(C)$ is a finite set of closed points. If we consider the unique decomposition
$C=C_1\cup\ldots\cup C_r$ such that the $\pi(C_i)$ are pairwise distinct points, then we have
$$\vol(C)=\sum_{i=1}^r\vol(C_i).$$
\end{prop}

\begin{proof}
If $\pi(C_i)=\{x_i\}$, then it is clear that 
$$\fb_m(C)=\bigcap_{i=1}^r\fb_m(C_j)$$
and $\fb_m(C_j)$ is cosupported at $x_j$ for every $m\geq 1$.
Therefore the assertion follows from Remark~\ref{ideals_cosupported_finite set}.
\end{proof}

\begin{prop}\label{bound_volume_by_multiplicity}
If $C\subseteq X_{\infty}(x)$, for some closed point $x\in X$, then 
$$\vol(X)\leq e_{x}(X).$$
\end{prop}

\begin{proof}
Note that if $\fm_x$ is the ideal defining $x$, then $\fm_x \subseteq \fb_1(C)$. Therefore
$\fm_x^p \subseteq \fb_1(C)^p \subseteq \fb_p(C)$ for every $p$, and we obtain
$\vol(C) \le e(\fm_x) = e_{x}(X)$.
\end{proof}

The following definition extends the notions of thin and fat arcs introduced in
\cite{ELM04,Ish05} to arbitrary sets of arcs. 

\begin{defi}
A subset $C$ of $X_\infty$ is said to be \emph{thin} 
if there exists a proper closed subscheme $Z \hookrightarrow X$
such that $C \subseteq Z_\infty$. If $C$ is not thin, then we say that $C$ is \emph{fat}. 
A subset $C$ of $X_{\infty}$ is a \emph{cylinder} if $C=\pi_m^{-1}(S)$ for some
$m$ and some constructible subset $S\subseteq X_m$, where $\pi_m\colon X_{\infty}\to X_m$
is the canonical projection. It is a basic fact that a cylinder $C$ is thin if and only if $C\subseteq (X_{\rm sing})_{\infty}$, where
$X_{\rm sing}$ is the singular locus of $X$ (see \cite[Lemma~5.1]{EM09}).
\end{defi}

\begin{prop}
\label{p:vol=0}
Let $C$ be a subset of $X_{\infty}$ whose image in $X$ is a finite set of closed points. If $C$ is thin, then $\vol(C)=0$, and if the closure of $C$ is a fat cylinder
and $X$ is analytically unramified at every point, then $\vol(C)>0$.
\end{prop}

\begin{proof}
Suppose first that there exists a proper closed subscheme $Z$ of $X$ such that
$C \subseteq Z_\infty$. Let $I_Z \subseteq \O_X$ be the ideal of $Z$. We have 
$I_Z \subseteq \fb_m(C)$ for every $m$, hence 
\[
\ell(\O_X/\fb_m(C)) = \ell(\O_Z/\fb_m(C)\O_Z) = o(m^n)
\]
since $\dim Z < n$. This implies that $\vol(C) = 0$. 

Let us assume now that $\ov C$ is a fat cylinder. Since $\ord_C=\ord_{\ov C}$ by Lemma~\ref{l:ov-C},
we may replace $C$ by $\ov C$ and thus assume that $C$ is closed. Since $C$ is a cylinder, 
it has finitely many irreducible components (see \cite[Proposition~3.5]{dFEI08}). One of these, say $C'$, has to be fat,
in which case $\ord_{C'}$ is a divisorial valuation
by \cite[Propositions~2.12 and~3.9]{dFEI08}. 
Of course, the image of $C'$ in $X$ consists of one closed point.
Using Propositions~\ref{p:C1-C2} and~\ref{p:ELS}, we conclude that
\[
\vol(C) \ge \vol(C') = \vol(\ord_{C'}) > 0.
\]
\end{proof}

We now address the results stated in the introduction. 
We begin with the first two propositions. 

\begin{proof}[Proof of Proposition~\ref{p:inclusion-exclusion}]
For every $p$, we have
\begin{equation}\label{eq3_prop_vol}
\fb_p(C_1\cup C_2)=\fb_p(C_1)\cap\fb_p(C_2)\quad\text{and}
\end{equation}
\begin{equation}\label{eq4_prop_vol}
\fb_p(C_1\cap C_2)\supseteq \fb_p(C_1)+\fb_p(C_2).
\end{equation}
The exact sequence
\[
0\to \O_X/(\fb_p(C_1)\cap\fb_p(C_2))\to \O_X/\fb_p(C_1)\oplus \O_X/\fb_p(C_2)
\to \O_X/(\fb_p(C_1)+\fb_p(C_2))\to 0
\]
implies 
\[
\ell(\O_X/\fb_p(C_1))+\ell(\O_X/\fb_p(C_2))
=\ell(\O_X/\fb_p(C_1)\cap\fb_p(C_2))+\ell(\O_X/\fb_p(C_1)+\fb_p(C_2)).
\]
Using \eqref{eq3_prop_vol} and \eqref{eq4_prop_vol}, we conclude
\[
\ell(\O_X/\fb_p(C_1))+\ell(\O_X/\fb_p(C_2))
\geq \ell(\O_X/\fb_p(C_1\cup C_2))+\ell(\O_X/\fb_p(C_1\cap C_2)).
\]
Then the assertion follows by dividing by $p^n/n!$ and letting $p$ go to infinity.
Note that this step uses the property that the limsup in the definition of the volume is, in fact, a limit.
\end{proof}

\begin{proof}[Proof of Proposition~\ref{p:vol-cont}]
Let $C_m={\rm Cont}^{\geq m}(\fa)$. It follows from definition that
$\fa^p\subseteq \fb_{mp}(C_m)$ for every $p \ge 1$. 
By \eqref{eq:vol=lim-e}, we have
\[
m^n\.\vol(C_m) = \lim_{p \to \infty} \frac{e(\fb_{mp}(C_m))}{p^n}
\le \lim_{p \to \infty} \frac{e(\fa^p)}{p^n} = e(\fa).
\]

Using the characterization of volume in Remark~\ref{vol_as_inf_e},
we deduce from the inclusion $\fa\subseteq \fb_m(C_m)$ that
\[
\vol(C_m)\leq \frac{e(\fb_m(C_m))}{m^n}\leq\frac{e(\fa)}{m^n}.
\]
Note that if $\gamma(t)\in C_m$, then $\gamma(t^p)\in C_{mp}$. This implies that we have an inclusion
\[
\fb_{mpq}(C_{mp})\subseteq\fb_{mq}(C_m)\quad\text{for every}\,\,q, 
\]
and therefore
\[
m^n\cdot \frac{e(\fb_{mq}(C_m))}{(mq)^n}\leq (mp)^n\cdot\frac{e(\fb_{mpq}(C_{mp}))}{(mpq)^n}.
\]
By letting $q$ go to infinity, we obtain 
\[
m^n\cdot\vol(C_m)\leq (mp)^n\cdot\vol(C_{mp}).
\]

In order to complete the proof, it is enough to show that when $m$ is divisible enough, we have $\vol(C_m)\geq\frac{e(\fra)}{m^n}$. 
Suppose that $E_1,\ldots,E_r$ are the divisors over $X$ corresponding to the Rees valuations associated to the ideal $\fra$
(see Example~\ref{Rees_valuations}). 
We put $q_i=\ord_{E_i}(\fa)$ and assume that $m$ is divisible by every $q_i$.
Recall that if $E$ is a divisor over $X$, then there is a sequence of irreducible closed subsets
$C^q_X(E)$, for $q\geq 1$, called the \emph{maximal divisorial sets}, which are defined as follows. 
If $\pi\colon Y\to X$ is a birational map such that $Y$ is smooth and $E$ is a smooth divisor on $Y$,
then $C^q_X(E)$ is the closure of $\pi_{\infty}(\Cont^{\geq q}(E))$. It is easy to see that 
$\ord_{C^q_X(E)}=q\cdot\ord_E$. For a discussion of these subsets of $X_{\infty}$, we refer to \cite{ELM04} and
\cite{dFEI08}. 
With this notation, we
consider the closed subset
\[
T_m:=\bigcup_{i=1}^r C^{m/q_i}_X(E_i).
\]
Note that we have $T_m\subseteq C_m$, hence
\[
\fb_{jm}(C_m)\subseteq\fb_{jm}(T_m)
=\bigcap_{i=1}^r\{f\in R\mid\ord_{E_i}(f)\geq jq_i\}=\overline{\fa^j},
\]
where we denote by $\overline{\fc}$ the integral closure of an ideal $\fc$. 
We conclude that
\[
e(\fb_{jm}(C_m))\geq e(\overline{\fa}^j)=j^n\cdot e(\fa).
\]
Dividing by $(jm)^n$ and letting $j$ go to infinity, we get $\vol(C_m)\geq \frac{e(\fa)}{m^n}$.
This completes the proof of the proposition.
\end{proof}

Next, we review the definition of jet-codimension and prove two more preliminary
properties before addressing the proof of Theorem~\ref{t:vol-codim}. 
Recall that the Krull codimension of a closed irreducible set $C \subseteq X_\infty$
is the dimension of the local ring $\O_{X_\infty,C}$, and is denoted by $\codim(C)$.
The definition extends to an arbitrary set $C \subseteq X_\infty$ by taking
the smallest codimenion of an irreducible component of the closure $\ov C$. 

While the Krull codimension is computed from the local rings of $X_\infty$, 
the jet-codimension is computed from the finite levels $X_m$. 
In order to define it, we need the following lemma. 

\begin{lem}\label{limit_jet_codim}
For every subset $C \subseteq X_\infty$, the limit
\[
\lim_{m \to \infty}\((m+1)n - \dim \ov{\p_m(C)}\)
\]
exists. 
\end{lem}

\begin{proof}
It follows from \cite[Lemma~4.3]{DL99} that for every $m$, the fibers of the map
$\pi_{m+1}(X_{\infty})\to \pi_m(X_{\infty})$ have dimension $\leq n$ (note that both sets are constructible 
by a result due to Greenberg \cite{Gre66}). It follows from Lemma~\ref{lem_constructible} below that
$\dim \overline{\pi_{m+1}(C)}\leq \dim \overline{\pi_m(C)}+n$, hence the sequence $(a_m)_{m\geq 1}$
with $a_m=(m+1)n - \dim \ov{\p_m(C)}$ is a non-decreasing sequence of integers. Therefore it either stabilizes
or it has limit infinity.
\end{proof}

\begin{lem}\label{lem_constructible}
Let $f\colon V\to W$ be a morphism of algebraic varieties over $k$ and suppose that $d$ is a non-negative integer and $A$ is a constructible subset
of $V$ such that for every $y\in f(A)$, we have $\dim(f^{-1}(y)\cap A)\leq d$. For every subset $B\subseteq A$, we have
$$\dim(\overline{B})\leq d+\dim(\overline{f(B)}).$$
\end{lem}

\begin{proof}
We can write $A=\bigcup_{i=1}^rA_i$, with each $A_i$ a locally closed subset of $V$. If $B_i=B\cap A_i$, then $B=\bigcup_{i=1}^rB_i$,
$\overline{B}=\bigcup_{i=1}^r\overline{B_i}$, and $\overline{f(B)}=\bigcup_{i=1}^r\overline{f(B_i)}$. Since it is enough to prove the assertion
for each $B_i$, it follows that we may assume that $A$ is a locally closed subset. In this case $A$ is open in $\overline{A}$, hence $A\cap\overline{B}$
is a dense open subset of $\overline{B}$. Since $\dim(\overline{B})=\dim(A\cap\overline{B})$ and the fibers of the morphism $A\cap\overline{B}\to \overline{f(B)}$
have dimension $\leq d$, we obtain the assertion in the lemma.
\end{proof}

\begin{defi}
The \emph{jet-codimension} of an irreducible closed subset $C$ of $X_\infty$ 
is defined to be
\[
\jetcodim(C) := \lim_{m \to \infty}\((m+1)n - \dim \ov{\p_m(C)}\).
\]
For an arbitrary subset $C\subseteq X_{\infty}$, we define $\jetcodim(C)$ to be the 
smallest jet-codimension of an irreducible component of $\overline{C}$.
\end{defi}

\begin{rmk}
It follows from the proof of Lemma~\ref{limit_jet_codim}
that if $C$ is closed and irreducible, then $\jetcodim(C)\geq n-\dim\overline{\pi(C)}\geq 0$.
This implies that for every $C\subseteq X$, we have $\jetcodim(C)\geq 0$. 
\end{rmk}

\begin{rmk}
If $C_1\subseteq C_2\subseteq X_{\infty}$, then $\jetcodim(C_1)\geq\jetcodim(C_2)$. Indeed,
if $C'_1$ is an irreducible component of $\overline{C_1}$, then there is an irreducible component $C'_2$ of $C_2$
such that $C'_1\subseteq C'_2$. In this case, for every $m$ we have
$$(m+1)n-\dim \overline{\pi_m(C'_1)}\geq (m+1)n-\dim \overline{\pi_m(C'_2)}.$$
By letting $m$ go to infinity, we conclude that $\jetcodim(C'_1)\geq\jetcodim(C'_2)\geq\codim(C_2)$.
Since this holds for every irreducible component of $\overline{C_1}$, we conclude that $\jetcodim(C_1)\geq\jetcodim(C_2)$. 
\end{rmk}

\begin{rmk}
For any subset $C \subseteq X_\infty$, 
we have $\codim(C) = \codim(\ov C)$ and $\jetcodim(C) = \jetcodim(\ov C)$.
\end{rmk}

If $X$ is smooth and $C\subseteq X_{\infty}$ is a cylinder, 
then $\jetcodim(C) = \codim(\p_m(C),X_m)$ for all $m \gg 1$.
As the next proposition shows, this is equal to
the Krull codimension $\codim(C)$.
More generally, we have the following property.

\begin{prop}
\label{p:codim}
If $X$ is smooth and $C \subseteq X_\infty$ is any set, 
then $\jetcodim(C) = \codim(C)$.
\end{prop}

\begin{proof}
The proof of the proposition follows immediately by applying the next lemma
to the irreducible components of $\ov C$. 
\end{proof}

\begin{lem}
If $X$ is smooth and $C \subseteq X_\infty$ is a closed irreducible subset, then 
\[
\jetcodim(C) = \codim(C),
\]
and this number is finite if and only if $C$ is a cylinder. 
\end{lem}

\begin{proof}
If $C$ is a cylinder, then it follows from \cite[Corollary~1.9]{ELM04} that
\[
\jetcodim(C) =  \codim(\p_m(C),X_m)=\codim(C) \quad \text{for}\,\, m \gg 1.
\]
Therefore it suffices to show that if $C$ is not a cylinder then 
\[
\jetcodim(C) = \dim(C) = \infty.
\]
In order to check this, consider the sequence of closed irreducible cylinders
\[
F_i := \p_i^{-1}(\ov{\p_i(C)}), \quad i \ge 0. 
\]
We have inclusions
\[
C \subseteq \dots \subseteq F_{i+1} \subseteq F_i \subseteq 
\dots \subseteq F_1 \subseteq F_0 \subseteq X_\infty.
\]
Moreover, since $C$ is closed, we have $C = \bigcap_{i \ge 0} F_i$.

Since $C$ is not a cylinder, the sequence $(F_i)_{i\geq 0}$ does not stabilize.
Therefore we can pick a subsequence $(F_{i_m})_{m\geq 0}$ such that
\[
C \subsetneq F_{i_m} \subsetneq F_{i_{m-1}} \subsetneq 
\dots \subsetneq F_{i_1} \subsetneq F_{i_0} \subsetneq X_\infty,
\]
which clearly implies that $\codim(C) = \infty$. 
In fact, for every $m$, if $p \geq i_m$, then we also have the sequence
\[
\overline{\p_p(C)} \subseteq \p_p(F_{i_m}) \subsetneq \p_p(F_{i_{m-1}}) \subsetneq 
\dots \subsetneq \p_p(F_{i_1}) \subsetneq \p_p(F_{i_0}) \subsetneq X_p.
\]
Note that for every $k\leq m$, the subset $\p_p(F_{i_k})$ of $X_p$ is irreducible and closed since $p \ge i_k$.
Therefore $\codim(\ov{\p_p(C)},X_p) \ge m$ and we conclude that
$\jetcodim(C) = \infty$.
\end{proof}

\begin{rmk}
The definition of jet-codimension generalizes to all sets the definition of 
codimension of a quasi-cylinder given in \cite{dFEI08}.
In general, if $X$ is singular and $C \subseteq X_\infty$ is a closed irreducible set, 
then there is only an inequality $\codim(C) \le \jetcodim(C)$ which can be strict
(e.g., see \cite[Example~2.8]{IR13}).
\end{rmk}

If $E$ is a prime exceptional divisor
over $X$ and $C_X^q(E) \subseteq X_\infty$ is the maximal divisorial set
associated to the divisorial valuation $q\.\ord_E$, then we have
\begin{equation}
\label{eq:^a_E=codim}
\jetcodim(C_X^q(E)) = q\.\^a_E(X)
\end{equation}
by \cite[Theorem~3.8]{dFEI08}.
Using this fact, it is easy to extend \cite[Corollary 0.2]{Mus02b} to the 
singular setting, as follows. 
This proposition is also proved in \cite[Proposition~3.5]{Ish13}, but since the
proof is short, we include it for the convenience of the reader. 

\begin{prop}
\label{p:Mat-lct-jets}
For every proper, nonzero ideal $\fa \subseteq R$ and every positive integer $m$, we have 
\[
\jetcodim(\Cont^{\ge m}(\fa)) \ge m\. \^\lct(\fa),
\]
with equality if $m$ is sufficiently divisible.
\end{prop}

\begin{proof}
By \cite[Propositions~3.5 and~2.12]{dFEI08}, $\Cont^{\ge m}(\fa)$ 
has finitely many fat irreducible components, and any such component $C$ is 
a maximal divisorial set. 
In particular, there is a fat irreducible component of the form
$C = C_X^q(E)$ for some divisorial valuation $q \.\ord_E$, such that
\[
\jetcodim(\Cont^{\ge m}(\fa)) = \jetcodim(C_X^q(E)) = q\.\^a_E(X),
\]
by \eqref{eq:^a_E=codim}. 
Note that $q\.\ord_E(\fa) \ge m$, since $C_X^q(E) \subseteq \Cont^{\ge m}(\fa)$.
On the other hand, we have
\[
\^\lct(\fa) \le \frac{\^a_E(X)}{\ord_E(\fa)}
\]
by the definition of Mather log canonical threshold. 
We conclude that $\jetcodim(\Cont^{\ge m}(\fa)) \ge m\. \^\lct(\fa)$.

On the other hand, suppose that $F$ is a divisor over $X$ such that $\^\lct(\fa) = \frac{\^a_F(X)}{\ord_F(\fa)}$
and suppose that $m=q\cdot\ord_F(\fa)$ for some positive integer $q$. In this case $C_X^q(F)\subseteq \Cont^{\geq m}(\fra)$,
hence 
$$\jetcodim(\Cont^{\geq m}(\fra))\leq \jetcodim(C_X^q(F))=q\cdot \^a_F(X)=m\cdot \^\lct(\fra).$$
By combining this with what we have already proved, we conclude that in this case we have 
$\jetcodim(\Cont^{\ge m}(\fa)) = m\. \^\lct(\fa)$.
\end{proof}

\begin{proof}[Proof of Theorem~\ref{t:vol-codim}]
For every $p\geq 1$, we have $C\subseteq {\rm Cont}^{\geq p}(\fb_p(C))$. 
Note that if $C$ lies over the closed point $x\in X$, defined by the maximal ideal $\fm_x$,
the ideal $\fb_p(C)$ is $\fm_x$-primary.
It follows from 
Proposition~\ref{p:Mat-lct-jets} that 
\begin{equation}\label{eq1_prop_volume}
\jetcodim(C)\geq \jetcodim\Cont^{\geq p}(\fb_p(C)))\ge p\cdot\^\lct(\fb_p(C)).
\end{equation}
On the other hand, Theorem~\ref{t:e-Mather-lct} implies that 
\begin{equation}\label{eq2_prop_volume}
\(n!\.\ell(\O_X/\fb_p(C))\)^{1/n}\.\^\lct(\fb_p(C))\geq n.
\end{equation}
By combining \eqref{eq1_prop_volume} and \eqref{eq2_prop_volume}, we get
\[
\(\frac{\ell(\O_X/\fb_p(C))}{p^n/n!}\)^{1/n}\.\jetcodim(C) \ge n.
\]
We conclude that 
\[
\vol(C)^{1/n}\.\jetcodim(C)
=\lim_{p\to\infty}\(\frac{\ell(\O_X/\fb_p(C))}{p^n/n!}\)^{1/n}\. \jetcodim(C)\geq n.
\]
This gives the first part of the statement of the theorem.
The second part follows from Proposition~\ref{p:codim}.
\end{proof}

\end{document}